\documentclass[11pt,reqno]{amsart}
\usepackage{tikz}
\usepackage{amsthm}
\usepackage{bm}
\usepackage{amsmath, bbold}
\usepackage{amsmath,amssymb}
\usepackage{subfig}
\usepackage{epstopdf}
\usepackage{epsfig}
\graphicspath{{./figures/}}
\usepackage{tikz}
\usepackage{multirow}
\usetikzlibrary{shapes,arrows}
\usepackage{bbold}
\tikzstyle{decision} = [diamond, draw, fill=blue!20, 
    text width=4.5em, text badly centered, node distance=3cm, inner sep=0pt]
\tikzstyle{block} = [rectangle, draw, fill=blue!20, 
    text width=5em, text centered, rounded corners, minimum height=4em]
\tikzstyle{line} = [draw, -latex']
\tikzstyle{cloud} = [draw, ellipse,fill=red!20, node distance=3cm,
    minimum height=2em]
 \newtheoremstyle{mystyle}{36pt}{}{}{}{\bfseries}{.}{ }{}
 
  \theoremstyle{plain}

 \newtheorem{definition}[]{Definition}
 
 \newtheorem{proposition}[]{Proposition}
 \newtheorem{example}{Example}

  \theoremstyle{remark}
 \newtheorem{remark}{Remark}

    \usetikzlibrary{positioning}
\tikzset{main node/.style={circle,fill=blue!20,draw,minimum size=1cm,inner sep=0pt},  }

\newcommand{\ts}{\mathsf{T}}

\newcommand{\dd}{\mathcal{\dagger}}

\newcommand{\vta}{\theta}

\newcommand{\D}{(\nabla_\omega\nabla_p)_{i,j}}

\providecommand{\bbs}[1]{\left(#1\right)}

\begin{document}
\title[Finite state Wasserstein common noises]{Langevin dynamics for the probability of {finite state} Markov processes}
\author[Li]{Wuchen Li}
\email{wuchen@mailbox.sc.edu}
\address{Department of Mathematics, University of South Carolina, 29208.}
\thanks{W. Li's work is supported by AFOSR MURI FP 9550-18-1-502, AFOSR YIP award No. FA9550-23-1-0087, NSF DMS-2245097, and NSF RTG: 2038080.}
\begin{abstract}
We study gradient drift-diffusion processes on a probability simplex set with finite state Wasserstein metrics, namely {\em { finite state Wasserstein common noises}}. A fact is that the Kolmogorov transition equation of finite reversible Markov processes satisfies the gradient flow of entropy in finite state Wasserstein space. This paper proposes to perturb finite state Markov processes with Wasserstein common noises. In this way, we introduce a class of stochastic reversible Markov processes. We also define stochastic transition rate matrices, namely Wasserstein Q-matrices, for the proposed stochastic Markov processes. We then derive the functional Fokker-Planck equation in the probability simplex, whose stationary distribution is a Gibbs distribution of entropy functional in a simplex set. {Several examples of Wasserstein drift-diffusion processes on a two-point state space are presented.} 
\end{abstract}
\maketitle
\section{Introduction}
Drift diffusions in probability density spaces play essential roles in macroscopic fluctuation theory, non-equilibrium statistical physics (e.g., glass dynamics), and stochastic evolutionary games \cite{Dean,FY, KK}. They describe stochastic behaviors of particles/agents/players perturbed by Brownian motions (common noises) on population states. A famous example is the Dean--Kawasaki equation (super Brownian motion) \cite{Dean, KK}. Nowadays, the Dean--Kawasaki equation has been shown as a gradient drift-diffusion in the Wasserstein-2 space \cite{KLR, WD1, WD}. In literature, gradient flows in Wasserstein-2 space form a class of density evolutionary equations \cite{am2006, vil2008}. Typical examples are heat equations, which are Wasserstein gradient flows of negative Boltzmann-Shannon entropy. 
While the Dean--Kawasaki equation adds ``Wasserstein common noises'' into these density evolutionary dynamics, {which} introduce a class of stochastic heat equations. 

Classical studies of Wasserstein drift diffusion processes are defined on a continuous domain, e.g., a $d$-dimensional torus. Not much has been studied on {finite states}, such as finite weighted graphs or equivalently reversible Markov chains. It has been shown that the gradient flow in finite state Wasserstein-2 spaces \cite{chow2012, EM1, M} characterizes the Kolomogrov forward equation of the reversible {finite state} Markov process \cite{EM1, M}. These generalized Wasserstein gradient flow belongs to {generalized Onsager's principles} \cite{Hanggi84, ON}. Many physical, {{chemical}} \cite{MM}, and social models, including stochastic evolutionary game theory \cite{FY,hofbauer1988theory}, are often studied on finite state spaces. Natural questions arise: 

{\em What are drift diffusion processes in finite state Wasserstein spaces? In particular, what are canonical Wasserstein common noises perturbed reversible Markov processes?} 

This paper presents Wasserstein type drift diffusion processes in a finite state simplex set. Following \cite{LiG}, we study the canonical diffusion process in finite state Wasserstein space. We then formulate over-damped Langevin dynamics in finite state Wasserstein spaces. We also present an example of the gradient drift-diffusion process. When the potential function is the $\phi$-divergence, and the activation function is the $\phi$-divergence induced mean function, the proposed SDE adds {\em geometric diffusions} in the transition equations of finite reversible Markov processes. In particular, we derive a {\em Wasserstein $Q$-matrix} function for modeling common and individual noises towards finite reversible Markov processes. Finally, numerical and analytical examples on a two-point space are introduced to illustrate the proposed Langevin dynamics in the probability simplex.  

In literature, gradient drift-diffusion processes in Wasserstein-2 space on continuous domain have been studied in \cite{WD3, WD2, WD1, WD}. In particular, a general Wasserstein gradient drift-diffusion process has been widely studied in \cite{KLR}, which satisfies the Dean--Kawasaki equation \cite{Dean, KK}. In fact, the Wasserstein common noise differs from the Larsy-Lions common noise {\cite{cardaliaguet2019master}}, while the later one is widely used in mean-field control and mean-field games \cite{PL1}. Meanwhile, \cite{CG} demonstrates that the generator of Larsy--Lion's common noise is only a partial Wasserstein Laplacian operator. In contrast to their works, we formulate Wasserstein common noises on finite state spaces, which is constructed from the Laplacian-Beltrami operator on finite state Wasserstein-2 space. However, not much is known on a finite state space involving Markov  processes. 
The finite state Markov process is essential in physical modeling and computations. In literature, \cite{EM1, M} define a class of discrete Wasserstein-2 metrics on finite state spaces. These Wasserstein type metrics depend on the average functions based on entropy functionals and transition rate functions of Markov processes. {We note that the finite state Wasserstein-2 metric also defines a Riemannian distance in the probability simplex. This Riemannian distance is in general different from the one defined in linear programming with a ground cost on finite states. See detailed studies in \cite{M}. The major issue is that the discrete Wasserstein-2 metric with different average functions can be viewed as the “discrete approximation” to the Wasserstein-2 metric in the continuous domain.} Using this framework, Wasserstein common noises added into Markov processes are natural classes of stochastic processes, which have vast range of applications, such as modeling dynamics from chemical reaction diffusion in {generalized Onsager's principles} \cite{Hanggi84, ON, MM}, finite state evolutionary games \cite{FY}, mean field games \cite{GLL}, and data sciences sampling problems \cite{WWG}. Mathematically, Wasserstein common noises on finite states also bring a class of challenging degenerate stochastic processes, whenever the process stays on the boundary of probability simplex set. We leave theoretical studies and numerical simulations of Wasserstein drift diffusions on discrete states in the future work. 

{It is also worth mentioning that finite state Wasserstein diffusion processes are closely related to, but different from Wright-Fisher diffusion processes, which are widely studied in information geometry \cite{IG} and population genetics \cite{FPG}. The Wright-Fisher diffusion is built from the Laplacian-Beltrami operator in the Fisher-Rao geometry, while the Wasserstein diffusion is built from the one in Wasserstein type geometry. The detailed modeling perspective for finite state Wasserstein diffusions are left in the future work. }
 
This paper is organized as follows. In section \ref{section2}, we briefly review the finite state Wasserstein-2 metric with gradient, divergence, and Laplacian operators. We next write the gradient-drift diffusion process on a probability simplex set. We also formulate the Fokker-Planck equation in finite state Wasserstein space. 
In section \ref{section3}, we present the modeling motivation of this paper. First, we review that the generator ($Q$-matrix) of the reversible Markov process is the gradient descent of divergence functions. We then add a stochastic perturbation into the finite reversible Markov  process and develop a Wasserstein $Q$-matrix for reversible Markov processes. Finally, several examples and numerical simulations of Wasserstein drift diffusions on a two-point space are presented in section \ref{section4}.

\section{Wasserstein common noises in probability simplex}\label{section2}
In this section, we formulate the canonical diffusion process in a discrete probability simplex set embedded with Wasserstein-2 metrics. We then formulate the gradient drift diffusion in probability simplex, which is a over-damped Wasserstein Langevin dynamics. 

\subsection{Finite state Wasserstein-2 space}
We review the Wasserstein-2 type metric on finite state sample space \cite{chow2012, EM1, M}; see also geometric computations in \cite{LiG}. {We also recommend readers about some discussions on graph operators in \cite{GLL}.}   

Consider a weighted undirected finite graph $G=\left(I, E, \omega\right)$, which contains the vertex set $I=\{1,\cdots, n\}$, the edge set $E$, and the weights set $\omega$. Here $\omega=(\omega_{ij})_{i,j\in I}\in \mathbb{R}^{n\times n}$ is a symmetric matrix, such that
$$\omega_{ij}=\begin{cases}\omega_{ji}>0
& \textrm{if $(i,j)\in E$;}\\
0 & \textrm{otherwise}.
\end{cases}$$ 
The set of neighbors or adjacent vertices of $i$ is denoted by $N(i)=\{j\in I \colon (i,j)\in E\}$. Define a {``volume'' vector} on weighted graph as $\pi=(\pi_i)_{i=1}^n$, such that
\begin{equation}\label{vol}
\pi_i:=\frac{\sum_{j\in N(i)}\omega_{ij}}{\sum_{(i,j)\in E}\omega_{ij}}. 
\end{equation}

We review gradient, divergence, and Laplacian operators on the graph $G$. Given a function $\Phi \colon I \to \mathbb{R}$, denote $\Phi=(\Phi_i)_{i=1}^n\in \mathbb{R}^n$. Define a weighted gradient as a function $\nabla_\omega \Phi \colon E \to \mathbb{R}$, 
\begin{equation*}
(i,j)  \, \mapsto \,\,  (\nabla_\omega\Phi)_{ij} :=\sqrt{\omega_{ij}}\,(\Phi_j-\Phi_i). 
\end{equation*} 
We call $\nabla_\omega\Phi$ a potential vector field on the edge set $E$. A general vector field is a anti-symmetric function on the edge set $E$, such that 
$v=\big( v_{ij} \big)_{(i,j)\in E}$,  
\begin{equation*}
v_{ij}=-v_{ji}, \quad (i,j) \in E. 
\end{equation*}
The divergence of a vector field $v$ is defined as a function $\mathrm{div}_\omega(v) \colon E \to \mathbb{R}$,
\begin{equation*}
i \, \mapsto \, \mathrm{div}_\omega(v)_i := \sum_{j\in N(i)}\sqrt{\omega_{ij}}\, v_{ij}.
\end{equation*}
{Here the divergence and gradient operators satisfy the ``discrete integration by parts'':
\begin{equation*}
\sum_{i=1}^n\Phi_i \mathrm{div}_\omega(v)_i=-\frac{1}{2}\sum_{(i,j)\in E} v_{ij}(\nabla_\omega\Phi)_{i,j}.
\end{equation*}
}
For a function $\Phi$ on $V$, the weighted graph Laplacian $\Delta_\omega \Phi \colon V \to \mathbb{R}$ satisfies 
$$
\Delta_\omega \Phi:=\mathrm{div}_\omega\bbs{\nabla_\omega\Phi }, \quad \text{ i.e., \, } i\, \mapsto \,
\Delta_\omega \Phi_i
= \sum_{j\in N(i)}\omega_{ij}\,(\Phi_j-\Phi_i).
$$
We use the convention that $\Delta_\omega\in \mathbb{R}^{n\times n}$ denotes a negative semi-definite matrix. {In other words, for any vector $\Phi\in\mathbb{R}^n$, 
\begin{equation*}
\Phi^{\ts}(\Delta_\omega\Phi)=\sum_{i=1}^n\Phi_i(\Delta_\omega\Phi_i)=\sum_{i=1}^n\Phi_i\sum_{j=1}^n\omega_{ij}(\Phi_j-\Phi_i)=-\frac{1}{2}\sum_{(i,j)\in E}\omega_{ij}(\Phi_i-\Phi_j)^2\leq 0. 
\end{equation*}

}
We next introduce the Wasserstein-2 type metric on a finite state. Denote the {open} simplex set as  
\begin{equation*}
\mathcal{P}(I) = \Big\{p=(p_i)_{i=1}^n\in \mathbb{R}^n \colon \sum_{i=1}^n p_i=1,\quad  p_i>0\Big\},
\end{equation*}
where $p$ is a probability vector and $p_i$ represents the discrete probability function on a node $i\in I$. We only study the interior of simplex set. {The restriction to the open set is
necessary for the later on construction of a diffusion process.} Denote the tangent space of $\mathcal{P}(I)$ at $p\in\mathcal{P}(I)$ as
\begin{equation*}
T_p\mathcal{P}(I) = \Big\{(\sigma_i)_{i=1}^n\in \mathbb{R}^n\colon  \sum_{i=1}^n\sigma_i=0 \Big\}.
\end{equation*}
{Define the following average function, also named activation function, $\theta: \mathbb{R}^+\times \mathbb{R}^+ \to \mathbb{R}^+$, where $\mathbb{R}^+=\{x\in\mathbb{R}^1\colon x\geq 0\}$ represents the nonnegative real number, such that}
\begin{itemize} 
\item[(i)]
\begin{equation*}
\theta(x, y)=\theta(y, x);  
\end{equation*}
\item[(ii)]
\begin{equation*}
\theta(x, y)> 0, \mbox{ if } xy\neq  0; 
\end{equation*}
{and 
\begin{equation*}
\theta(x,y)=0, \mbox{ if } xy=0;
\end{equation*}}
\item[(iii)]
\begin{equation*}
\theta(x, y)\in C^{2}(\mathbb{R}^+, \mathbb{R}^+). 
\end{equation*}
\end{itemize}
There are many choices of average functions; see \cite{EM1, ON}. 
{\begin{example}[Arithmetic mean] 
{Suppose (i), (iii) hold. Consider}
$$\theta(x,y)=\frac{x+y}{2}.$$
\end{example}}
\begin{example}[Geometric mean]
{Suppose (i), (ii), (iii) hold. Consider}
$$\theta(x,y)=\sqrt{xy}.$$
\end{example}
\begin{example}[Harmonic mean]
{Suppose (i), (ii), (iii) hold. Consider}
$$\theta(x,y)=\frac{1}{\frac{1}{x}+\frac{1}{y}}.$$
\end{example}
\begin{example}[Logarithm mean]
{Suppose (i), (ii), (iii) hold. Consider}
$$\theta(x,y)=\frac{x-y}{\log x-\log y}.$$
\end{example}
\begin{example}[$\phi'$ mean]
{Suppose (i), (ii), (iii) hold. Consider}
$$\theta(x,y)=\frac{x-y}{\phi'(x)-\phi'(y)},$$
where $\phi\in C^{1}(\mathbb{R}; \mathbb{R})$ is a convex function with $\phi(1)=0$. If $\phi(x)=x\log x$, then $\phi'(x)=\log x$ and the $\phi'$ mean recovers the logarithm mean. 
\end{example}
{Under the notation of average function, we define the following weighted Laplacian matrix, which depends on the probability $p$ on a simplex set.}
\begin{definition}[Probability weighted Laplacian matrix]
Denote $L(p)=(L(p)_{ij})_{1\leq i,j\leq n}\in\mathbb{R}^{n\times n}$, such that 
\begin{equation*}
L(p)_{ij}:=\begin{cases}
-\omega_{ij}\theta_{ij}(p)&\textrm{if $j\neq i$};\\
\sum_{k\in N(i)}\omega_{ki}\theta_{ki}(p)&\textrm{if $j=i$,}
\end{cases}
\end{equation*}
where $\theta_{ij}$ is an average function defined as 
\begin{equation*}
\theta_{ij}(p):=\theta\left(\frac{p_i}{\pi_i}, \frac{p_j}{\pi_j}\right), \quad \textrm{for any $i$, $j\in I$.}
\end{equation*}
From now on, we call $L(p)$ the probability weighted Laplacian matrix. 
\end{definition}
{We also use the following notation to represent the probability weighted matrix $L(p)$. 
Denote a matrix function $\theta(p)=(\theta_{ij}(p))_{1\leq i, j\leq n}\in\mathbb{R}^{n\times n}$. Denote a vector field $\theta(p)\nabla_\omega\Phi\colon E\rightarrow\mathbb{R}$ as 
\begin{equation*}
(\theta(p) \nabla_\omega \Phi)_{ij}:=\theta_{ij}(p)(\nabla_\omega\Phi)_{ij}=\theta_{ij}(p)\sqrt{\omega_{ij}}(\Phi_j-\Phi_i). 
\end{equation*}
Clearly, $\theta(p)\nabla_\omega\Phi$ is a vector field on the edge set $E$, such that
\begin{equation*}
(\theta(p) \nabla_\omega \Phi)_{ij}=-(\theta(p) \nabla_\omega \Phi)_{ji}. 
\end{equation*}
We also write 
\begin{equation*}
L(p):=-\mathrm{div}_\omega (\theta(p)\nabla_\omega)=-\mathrm{div}_\omega (\theta \nabla_\omega). 
\end{equation*}
This means that for any vector $\Phi\in \mathbb{R}^n$, 
\begin{equation*}
(L(p)\Phi)_i=-\mathrm{div}_\omega (\theta(p)\nabla_\omega \Phi)_i=-\sum_{j\in N(i)}{\omega_{ij}}\theta_{ij}(p)(\Phi_j-\Phi_i). 
\end{equation*}
We remark that $L(p)$ is a symmetric nonegative definite matrix with the row sum zero condition. In other words,  for any testing vector $\Phi\in \mathbb{R}^n$, 
\begin{equation*}
\Phi^{\ts}L(p)\Phi=-\sum_{i=1}^n\Phi_i\mathrm{div}_\omega (\theta(p)\nabla_\omega \Phi)_i=\frac{1}{2}\sum_{(i,j)\in E}\omega_{ij}(\Phi_i-\Phi_j)^2\theta_{ij}(p)\geq 0, 
\end{equation*}
where we note that $\theta_{ij}(p)\geq 0$ from the definition of average function $\theta$. Denote $\mathbb{1}=(1,\cdots, 1)^{\ts}$, then 
\begin{equation*}
\mathbb{1}^{\ts}L(p)\Phi=-\sum_{i=1}^n\mathrm{div}_\omega (\theta(p)\nabla_\omega \Phi)_i=\frac{1}{2}\sum_{(i,j)\in E}(\nabla_\omega \mathbb{1})_{ij}(\nabla_\omega\Phi)_{ij}\theta_{ij}(p)=0. 
\end{equation*}

We now briefly study the property of matrix $L(p)$. When $\theta_{ij}(p)>0$ for all $(i,j)\in E$, matrix $L(p)$ is with exactly one zero eigenvalue. The corresponding unit eigenvector is the vector $\frac{1}{\sqrt{n}}(1,\cdots, 1)^{\ts}$. This is true since $\Phi^{\ts}L(p)\Phi=0$ has only one linear independent solution, 
\begin{equation*}
\Phi_i=\Phi_j, \quad \textrm{for any $(i,j)\in E$}. 
\end{equation*}
Since the graph is connected, we have $\Phi_1=\Phi_2=\cdots=\Phi_n$.} Thus, if $p$ stays in the interior of probability simplex, the diagonalization of $L(p)$ satisfies  
\begin{equation*}
L(p)=U(p)\begin{pmatrix}
0 & & &\\
& {\lambda_{1}(p)}& &\\
& & \ddots & \\
& & & {\lambda_{n-1}(p)}
\end{pmatrix}U(p)^{\ts},
\end{equation*}
where $0<\lambda_1(p)\leq\cdots\leq \lambda_{n-1}(p)$ are eigenvalues of $L(p)$ in the ascending order, 
and $U(p)=(u_0(p),u_1(p),\cdots, u_{n-1}(p))\in \mathbb{R}^{n\times n}$ is the orthogonal matrix of eigenvectors, with $u_0=\frac{1}{\sqrt{n}}(1,\cdots, 1)^{\ts}$. We also denote the pseudo-inverse of $L(p)$ as $L(p)^{\dagger}$, such that  
\begin{equation*}
L(p)^{\dagger}=U(p)\begin{pmatrix}
0 & & &\\
& \frac{1}{\lambda_{1}(p)}& &\\
& & \ddots & \\
& & & \frac{1}{\lambda_{n-1}(p)}
\end{pmatrix}U(p)^{\ts} .
\end{equation*}
{Using the probability weighted Laplacian matrix $L(p)$,} the finite state Wasserstein-2 metric is defined as follows. 
\begin{definition}[Finite state Wasserstein-2 metric]
The inner product $g^W:\mathcal{P}(I)\times T_p\mathcal{P}(I)\times T_p\mathcal{P}(I)\rightarrow\mathbb{R}$ is given as 
\begin{equation*}
\begin{split}
g^W(p)(\sigma_1,\sigma_2):=&{\sigma_1}^{\ts}L(p)^{\dagger}\sigma_2=\Phi_1^{\ts}L(p)\Phi_2\\
=&\frac{1}{2}\sum_{(i,j)\in E}(\nabla_\omega \Phi_1)_{ij}(\nabla_\omega \Phi_2)_{ij}\theta_{ij}(p),
\end{split}
\end{equation*}
where we define a vector $\Phi$ up to constant shrift satisfying 
 $$\sigma_k=L(p)\Phi_k=-\mathrm{div}_\omega (\theta\nabla_\omega \Phi_k)\in T_p\mathcal{P}(I), \quad k=1,2.$$ 
The inner product $g^W$ defines a Wasserstein-2 metric on the simplex set $\mathcal{P}(I)$. From now on, we name $(\mathcal{P}(I), g^W)$ the {\em probability manifold}. 
\end{definition}
\begin{remark}
{We remark that the inner product is zero on a constant vector of $\Phi$, i.e. $\Phi=c u_0$, where $c\in \mathbb{R}$. If one defines $\sigma_k=L(p)\Phi_k$, $k=1,2$, the scalar product $g^W$ is mapped to the tangent space of the open probability simplex and becomes a $(n-1)$ dimensional Riemannian metric. In particular, the following finite dimensional duality relation holds: $\sigma_k=L(p)\Phi_k$. We find a solution $\Phi_k:=L(p)^{\dd}\sigma_k$. In general, $\Phi_k$ can be written as $L(p)^{\dd}\sigma_k+cu_0$, up to a constant vector $u_0$ shrift. This is true since 
\begin{equation*}
\sigma\in \mathrm{Range}(L(p))=\mathrm{span}\{u_1,u_2,\cdots, u_{n-1}\}.
\end{equation*}
}
\end{remark}

We last present gradient, divergence, and Laplace-Beltrami operators in the probability manifold $(\mathcal{P}(I), g^W)$.
 The volume form in $(\mathcal{P}(I), g^W)$ satisfies 
\begin{equation*}
d\textrm{vol}_{W}:={\mathbb{\Pi}(p)}^{-\frac{1}{2}}dp,\quad\textrm{with}\quad \mathbb{\Pi}(p):=\Pi_{i=1}^{n-1}\lambda_i(p),
\end{equation*}
where $\lambda_i(p)$ are positive eigenvalues of the matrix function $L(p)$ and $dp$ is the  {volume form in the simplex set. Denote $\nabla_p$, $\nabla_p\cdot$ as gradient, divergence operators in $\mathbb{R}^n$.}
{We refer readers to check definitions of Riemannian operators on a simplex set in \cite{IG}.}  
\begin{proposition}
Denote $\mathbb{F}\in C^{\infty}(\mathcal{P}(I); \mathbb{R})$, and denote a vector function $\mathbb{H}=(\mathbb{H}_i)_{i=1}^n\in C^{\infty}(\mathbb{P}(I); \mathbb{R}^n)$.
\begin{itemize}
\item[(i)] The gradient operator $\textrm{grad}_W\colon C^{\infty}(\mathcal{P}(I); \mathbb{R})\rightarrow C^{\infty}(\mathcal{P}(I); \mathbb{R}^n)$ satisfies 
\begin{equation*}
\begin{split}
\textrm{grad}_W\mathbb{F}(p)=&L(p)\nabla_p\mathbb{F}(p)\\
=&\Big(-\mathrm{div}_\omega (\theta\nabla_\omega \nabla_p\mathbb{F}(p))_i\Big)_{i=1}^n\\
=&\Big(-\sum_{j\in N(i)}\sqrt{\omega_{ij}}\vta_{ij}(p)(\nabla_\omega \nabla_p)_{i,j}\mathbb{F}(p)\Big)_{i=1}^n,
\end{split}
\end{equation*}
where $$(\nabla_\omega \nabla_p)_{i,j}\mathbb{F}(p):=\sqrt{\omega_{ij}}(\frac{\partial}{\partial p_j}-\frac{\partial}{\partial p_i})\mathbb{F}(p).$$

\item[(ii)] The divergence operator $\textrm{div}_W\colon C^{\infty}(\mathcal{P}(I); \mathbb{R}^n)\rightarrow C^{\infty}(\mathcal{P}(I);\mathbb{R})$ satisfies
\begin{equation*}
\begin{split}
\textrm{div}_{{W}}\mathbb{H}(p)=&\mathbb{\Pi}(p)^{\frac{1}{2}}\nabla_p\cdot\Big(\mathbb{\Pi}(p)^{-\frac{1}{2}}\mathbb{H}(p)\Big). 
\end{split}
\end{equation*}
 
\item[(iii)] The Laplace-Beltrami operator $\Delta_W\colon C^{\infty}(\mathcal{P}(I);\mathbb{R})\rightarrow C^{\infty}(\mathcal{P}(I);\mathbb{R})$ satisfies 
\begin{equation*}
\begin{split}
\Delta_W \mathbb{F}(p)=&\textrm{div}_W(\textrm{grad}_W\mathbb{F}(p))\\
=&\mathbb{\Pi}(p)^{\frac{1}{2}}\nabla_p\cdot\Big(\mathbb{\Pi}(p)^{-\frac{1}{2}} L(p)\nabla_p\mathbb{F}(p)\Big)\\
=&-\frac{1}{4}\sum_{(i,j)\in E}(\nabla_\omega \nabla_p)_{i,j}\mathbb{F}(p)(\nabla_\omega\nabla_p)_{i,j}\log\mathbb{\Pi}(p) \vta_{ij}(p)\\
&+\frac{1}{2}\sum_{(i,j)\in E}(\nabla_\omega\nabla_p)_{i,j}(\nabla_\omega \nabla_p)_{i,j} \mathbb{F}(p)\vta_{ij}(p)\\
&+\frac{1}{2}\sum_{(i,j)\in E}(\nabla_\omega\nabla_p)_{i,j}\mathbb{F}(p)(\nabla_\omega \nabla_p)_{i,j} \vta_{ij}(p),
\end{split}
\end{equation*}
where 
\begin{equation*}
\begin{split}
(\nabla_\omega \nabla_p)_{i,j}(\nabla_\omega \nabla_p)_{i,j}\mathbb{F}(p):=&(\sqrt{\omega_{ij}}(\frac{\partial}{\partial p_j}-\frac{\partial}{\partial p_i}))^2\mathbb{F}(p)\\
=&\omega_{ij}(\frac{\partial^2}{\partial p_i^2}-2\frac{\partial^2}{\partial p_i\partial p_j}+\frac{\partial^2}{\partial p_j^2})\mathbb{F}(p),
\end{split}
\end{equation*}
and 
\begin{equation*}
(\nabla_\omega\nabla_p)_{i,j}\theta_{ij}(p):=\sqrt{\omega_{ij}}(\frac{\partial}{\partial p_j}-\frac{\partial}{\partial p_i})\theta_{ij}(p).
\end{equation*}
\end{itemize}
\end{proposition}
{
\begin{proof}
The derivation of gradient, divergence, and Laplace-Beltrami operators follows from the proof in \cite[Proposition 1 and 8]{LiG}. We omit them here for the simplicity of presentation.  

\end{proof}}
\subsection{Finite state cannocial Wasserstein common noises}
We are ready to introduce a canonical diffusion process on a manifold  $(\mathcal{P}(I),g^W)$. 
\begin{definition}[Wasserstein common noises on graphs]
Consider an Ito stochastic differential equation  
\begin{equation}\label{GSDE}
dp_t=\mathrm{div}_\omega(\theta(p_t) \nabla_\omega \nabla_p\log\frac{\mathbb{\Pi}(p_t)^{\frac{1}{2}}}{\theta(p_t)})dt+\sqrt{2}\mathrm{div}_\omega(\sqrt{\theta(p_t)}dB^E_t),
\end{equation}
where {$p_0=p(0)\in\mathcal{P}(I)$ is an initial value probability function}, {$p_t=p(t)\in\mathbb{R}^n$} is the solution of SDE \eqref{GSDE}, $B_t^E:=(B^E_{ij}(t))_{1\leq i,j\leq n}$ with $B_{ij}^E(t)=B_{ij}^E=\frac{1}{\sqrt{2}}(B_{ij}-B_{ji})$, and $B_{ij}$, $1\leq i,j\leq n$, are standard independent Brownian motions in $\mathbb{R}^{n\times n}$ with zero means and unity rate variances. In details, for any $i\in I$, equation \eqref{GSDE} satisfies 
\begin{equation*}
\begin{split}
dp_i(t)=&~~\sum_{j\in N(i)}\sqrt{\omega_{ij}}(\nabla_\omega \nabla_p)_{i,j}\log\frac{\mathbb{\Pi}(p(t))^\frac{1}{2}}{\theta_{ij}(p(t))}\theta_{ij}(p(t)) dt\\
&+\sum_{j\in N(i)}\sqrt{\omega_{ij}\theta_{ij}(p(t))} (dB_{ij}(t)-dB_{ji}(t)),
\end{split}
\end{equation*}
where
\begin{equation*}
(\nabla_\omega \nabla_p)_{i,j} \log\frac{\mathbb{\Pi}(p)^{\frac{1}{2}}}{\theta(p)}:=\sqrt{\omega_{ij}}(\frac{\partial}{\partial p_j}-\frac{\partial}{\partial p_i})(\log\frac{\mathbb{\Pi}(p)^{\frac{1}{2}}}{\theta_{ij}(p)}).
\end{equation*}
We call a solution of \eqref{GSDE} the $\sqrt{2}$-Wasserstein common noise on finite states.
\end{definition}
{We note that the above SDE is defined on $\mathbb{R}^n$, whose solution stays in the simplex set when the initial value $p(0)$ stays in the simplex set. This is true because the discrete divergence operator has the following property. For any discrete vector fields $v_{ij}=-v_{ij}$, we have
\begin{equation*}
\sum_{i=1}^n\mathrm{div}_\omega(v)_i= \sum_{i=1}^n\sum_{j\in N(i)}\sqrt{\omega_{ij}}v_{ij}=\frac{1}{2}\sum_{(i,j)\in E}\sqrt{\omega_{ij}}(v_{ij}+v_{ji})=0.
\end{equation*}
We check that $\theta(p_t) \nabla_\omega \nabla_p\log\frac{\mathbb{\Pi}(p_t)^{\frac{1}{2}}}{\theta(p_t)}$, $\sqrt{\theta(p_t)}dB^E_t$ are discrete vector fields. I.e., 
\begin{equation*}
(\theta(p_t) \nabla_\omega \nabla_p\log\frac{\mathbb{\Pi}(p_t)^{\frac{1}{2}}}{\theta(p_t)})_{ij}=-(\theta(p_t) \nabla_\omega \nabla_p\log\frac{\mathbb{\Pi}(p_t)^{\frac{1}{2}}}{\theta(p_t)})_{ji},
\end{equation*}
and 
\begin{equation*}
(\sqrt{\theta(p_t)}dB^E_t)_{ij}=-(\sqrt{\theta(p_t)}dB^E_t)_{ji}.
\end{equation*}
Thus  
\begin{equation*}
\sum_{i=1}^n dp_{it}=\sum_{i=1}^n \Big\{\mathrm{div}_\omega(\theta(p_t) \nabla_\omega \nabla_p\log\frac{\mathbb{\Pi}(p_t)^{\frac{1}{2}}}{\theta(p_t)})_idt+\sqrt{2}\mathrm{div}_\omega(\sqrt{\theta(p_t)}dB^E_t)_i\Big\}=0. 
\end{equation*}
This explains that $\sum_{i=1}^n p_i(t)=\sum_{i=1}^np_i(0)$. 
 }

We next present both Kolmogorov forward and backward operators for SDE \eqref{GSDE}. 
\begin{proposition}[Kolmogorov operators in probability manifold]\label{GKO}
Denote the probability density function and the test function on the simple set $\mathcal{P}(I)$ as
\begin{equation*}
\mathbb{P}(p)\in C^{\infty}(\mathcal{P}(I); \mathbb{R}),\qquad \mathbb{\Phi}(p)\in C^{\infty}(\mathcal{P}(I); \mathbb{R}).
\end{equation*}
We denote $C^{\infty}(\mathcal{P}(I); \mathbb{R})$ as the set of smooth functions, whose domain is the simplex set. $\mathbb{P}(p)$ is a probability density function supported at the simplex set $\mathcal{P}(I)$, and $\mathbb{\Phi}(p)$ is a function, whose domain is the simplex set $\mathcal{P}(I)$.

Then the Kolmogorov forward operator of SDE \eqref{GSDE} satisfies  
\begin{equation*}
\begin{split}
\mathsf{L}^*_W\mathbb{P}(p)=&\frac{1}{2}\nabla_p\cdot \Big(\mathbb{P}(p) L(p) \nabla_p\log\mathbb{\Pi}(p)\Big)+\nabla_p\cdot \Big(L(p) \nabla_p\mathbb{P}(p)\Big)\\
=&\frac{1}{2} \Big(\nabla_p\mathbb{P}(p), L(p) \nabla_p\log\mathbb{\Pi}(p)\Big)+\frac{1}{2}\mathbb{P}(p)\nabla_p\cdot\Big(L(p)\nabla_p\log\mathbb{\Pi}(p)\Big)+\nabla_p\cdot \Big(L(p) \nabla_p\mathbb{P}(p)\Big).
\end{split}
\end{equation*}
And the Kolmogorov backward operator of SDE \eqref{GSDE} satisfies 
\begin{equation*}
\begin{split}
\mathsf{L}_W\mathbb{\Phi}(p)=&-\frac{1}{2}(\nabla_p\mathbb{\Phi}(p), L(p)\nabla_p\log\mathbb{\Pi}(p))+\nabla_p\cdot(L(p)\nabla_p\mathbb{\Phi}(p)).
\end{split}
\end{equation*}
In details, 
\begin{equation*}
\begin{split}
\mathsf{L}^*_W\mathbb{P}(p)=&\quad\frac{1}{4}\sum_{(i,j)\in E}  (\nabla_\omega \nabla_p)_{i,j} \mathbb{P}(p) (\nabla_\omega \nabla_p)_{i,j}\log\mathbb{\Pi}(p)\vta_{ij}(p)\\
&+\frac{1}{4}\mathbb{P}(p)\sum_{(i,j)\in E}\D\log\mathbb{\Pi}(p)\D\theta_{ij}(p)\\
&+\frac{1}{4}\mathbb{P}(p)\sum_{(i, j)\in E}(\nabla_\omega\nabla_p)_{i,j}(\nabla_\omega \nabla_p)_{i,j} \log\mathbb{\Pi}(p)\vta_{ij}(p)\\
&+\frac{1}{2}\sum_{(i,j)\in E}(\nabla_\omega\nabla_p)_{i,j}(\nabla_\omega \nabla_p)_{i,j} \mathbb{P}(p)\vta_{ij}(p)\\
&+\frac{1}{2}\sum_{(i,j)\in E} (\nabla_\omega\nabla_p)_{i,j}\mathbb{P}(p)(\nabla_\omega\nabla_p)_{i,j}\theta_{ij}(p),
\end{split}
\end{equation*}
and 
\begin{equation*}
\begin{split}
\mathsf{L}_W\mathbb{\Phi}(p)=&-\frac{1}{4}\sum_{(i,j)\in E}   (\nabla_\omega \nabla_p)_{i,j}\mathbb{\Phi}(p)(\nabla_\omega \nabla_p)_{i,j} \log\mathbb{\Pi}(p)\vta_{ij}(p)\\
&+\frac{1}{2}\sum_{(i,j)\in E}(\nabla_\omega\nabla_p)_{i,j}(\nabla_\omega \nabla_p)_{i,j} \mathbb{\Phi}(p)\vta_{ij}(p)\\
&+\frac{1}{2}\sum_{(i,j)\in E}(\nabla_\omega\nabla_p)_{i,j}\mathbb{\Phi}(p)(\nabla_\omega \nabla_p)_{i,j} \vta_{ij}(p).
\end{split}
\end{equation*}
\end{proposition}
The derivations of $L_W^*$ and $L_W$ are provided in appendix. 
\subsection{Langevin dynamics in finite state Wasserstein space}
We next derive the overdamped Langevin dynamics in finite state Wasserstein space. It describes gradient drift diffusion processes in the probability simplex set.
{Here the gradient drift diffusion process refers to the time reversible stochastic process, which means that it satisfies the detailed balance condition. See the definition in \cite[Section 4.6]{GP}.}

\begin{proposition}[Gradient drift diffusion processes in probability simplex]\label{prop2}
Given $\mathbb{V}\in C^{\infty}(P(I); \mathbb{R})$, consider the gradient drift diffusion process 
\begin{equation}\label{GGSDE}
dp_t=\mathrm{div}_\omega(\theta(p_t) \nabla_\omega \nabla_p[\mathbb{V}(p_t)+\beta\log\frac{\mathbb{\Pi}(p_t)^{\frac{1}{2}}}{\theta(p_t)}])dt+\sqrt{2\beta}\mathrm{div}_\omega(\sqrt{\theta(p_t)}dB^E_t),
\end{equation}
where $\beta>0$ is a scalar. In details, for any $i\in I$, equation \eqref{GGSDE} satisfies 
\begin{equation*}
\begin{split}
dp_i(t)=&\qquad\sum_{j\in N(i)}\sqrt{\omega_{ij}}(\nabla_\omega \nabla_p)_{i,j}\Big(\mathbb{V}(p(t))+\beta\log\frac{\Pi(p(t))^{\frac{1}{2}}}{\theta_{ij}(p(t))}\Big)\theta_{ij}(p(t))dt\\
&+\sqrt{\beta}\sum_{j\in N(i)}\sqrt{\omega_{ij}\theta_{ij}(p(t))} (dB_{ij}(t)-dB_{ji}(t)).
\end{split}
\end{equation*}
The Fokker-Planck equation of SDE \eqref{GGSDE} satisfies 
\begin{equation}\label{FPE}
\begin{split}
\frac{\partial}{\partial t}\mathbb{P}(t,p)=&\nabla_p\cdot(\mathbb{P}(t,p)L(p)\nabla_p\mathbb{V}(p))+\beta\mathsf{L}^*_W\mathbb{P}(t,p), 
\end{split}
\end{equation}
where the solution $\mathbb{P}(t, p)$ represents the probability density function of SDE \eqref{GGSDE}. 
Assume that $Z:=\int_{\mathcal{P}(I)}e^{-\frac{1}{\beta}\mathbb{V}(p)}\mathbb{\Pi}(p)^{-\frac{1}{2}}dp<+\infty$. Then the stationary solution of equation \eqref{FPE} satisfies 
\begin{equation*}
\mathbb{P}^*(p)=\frac{1}{Z}e^{-\frac{1}{\beta}\mathbb{V}(p)}\mathbb{\Pi}(p)^{-\frac{1}{2}}.
\end{equation*}

\end{proposition}
\begin{proof}
The Kolmogorov forward equation of SDE \eqref{GGSDE} satisfies  
\begin{equation*}
\begin{split}
\frac{\partial \mathbb{p}(t,p)}{\partial t}=&\textrm{div}_{W}(\mathbb{p}(t,p)\textrm{grad}_W\mathbb{V}(p))+\beta\Delta_{W}\mathbb{p}(t,p)\\
=&\mathbb{\Pi}(p)^{\frac{1}{2}}\nabla_p\cdot\Big(L(p)\big[\mathbb{p}(t,p)\nabla_p\mathbb{V}(p)+ \beta\nabla_p\mathbb{p}(t,p)\big] \mathbb{\Pi}(p)^{-
\frac{1}{2}}\Big).
\end{split}
\end{equation*}
Again, denote $\mathbb{P}(t,p)=\mathbb{p}(t,p)\mathbb{\Pi}(p)^{-\frac{1}{2}}$, then we have
\begin{equation*}
\begin{split}
\frac{\partial\mathbb{P}(t,p)}{\partial t}
=&\nabla_p\cdot(\mathbb{P}(t,p)L(p)\nabla_p\mathbb{V}(p))+\beta\mathsf{L}^*_W\mathbb{P}(t,p)\\
=&\beta\nabla_p\cdot(\mathbb{P}(t,p)L(p)\nabla_p\log\frac{\mathbb{P}(t,p)}{e^{-\frac{1}{\beta}\mathbb{V}(p)}\mathbb{\Pi}(p)^{-\frac{1}{2}}}).
\end{split}
\end{equation*}
This finishes the proof. 
\end{proof}
\begin{remark}
We note that the dynamical behaviors of SDEs \eqref{GSDE} or \eqref{GGSDE} are often complicated when $p_i$, $p_j$ are close to zero. They are degenerate SDEs on the boundary point of the simplex set. In modeling of finite state population games, we need to construct some reflecting boundary conditions to ensure the wellposedness of SDE \eqref{GGSDE}. We leave their studies in future works.  
\end{remark}
\section{Stochastic reversible Markov processes}\label{section3}
In this section, we present an important example of gradient drift diffusion process \eqref{GGSDE}. This is the main result of this paper. 
We first review the fact that gradient flows in $(\mathcal{P}(I), g^W)$ {characterize Kolmogorov forward equations} for finite state reversible Markov processes. In other words, there exists a $Q$-matrix, the generator of finite reversible Markov process, which is a gradient descent direction of {relative entropy} in $(\mathcal{P}(I), g^W)$. We next demonstrate that the proposed SDE adds {\em geometric diffusions} in transition equations of finite reversible Markov processes. In particular, we derive a Wasserstein $Q$-matrix function for modeling {both common noises and individual noises} towards finite state reversible Markov processes.  

In this section, we always consider an activation function: \begin{equation*}
\theta(x,y)=\frac{x-y}{\phi'(x)-\phi'(y)}, 
\end{equation*}
where $\phi\in C^{1}(\mathbb{R};\mathbb{R})$ is a convex function with $\phi(1)=0$. Let the functional $\mathbb{V}$ in equation \eqref{GGSDE} be the $\phi$-divergence:
\begin{equation*}
\mathbb{V}(p)=\mathrm{D}_\phi(p\|\pi):=\sum_{i=1}^n\phi(\frac{p_i}{\pi_i})\pi_i,
\end{equation*}
where $\pi\in\mathbb{R}^n$ is defined in \eqref{vol}. One example of $\phi$-divergence is the Kullback--Leibler (KL) divergence. E.g., when $\phi(x)=x\log x$, then $\mathrm{D}_\phi(p\|\pi)=\mathrm{D}_{\mathrm{KL}}(p\|\pi)=\sum_{i=1}^np_i\log\frac{p_i}{\pi_i}$. 
\subsection{Reversible Markov process}
We first review that gradient flows of $\phi$-divergences in $(\mathcal{P}(I), g^W)$ form reversible Markov processes; shown in \cite{EM1, M} and strong Onsager gradient flows \cite{ON}. In other words, let $\beta=0$. In this case, SDE \eqref{GGSDE} satisfies an ordinary differential equation, which is the gradient flow of $\phi$-divergence in $(\mathcal{P}(I), g^W)$: 
\begin{equation}\label{GD}
\begin{split}
\frac{dp_i(t)}{dt}=&\mathrm{div}_\omega (\theta(p(t))\nabla_\omega \nabla_p\mathrm{D}_{\phi}(p(t)\|\pi))_i\\
=&\sum_{j\in N(i)}\omega_{ij}\theta_{ij}(p(t))(\frac{\partial}{\partial p_j}-\frac{\partial}{\partial p_i})\mathrm{D}_{\phi}(p(t)\|\pi)\\
=&\sum_{j\in N(i)}\omega_{ij}\frac{\frac{p_j(t)}{\pi_j}-\frac{p_i(t)}{\pi_i}}{\phi'(\frac{p_j(t)}{\pi_j})-\phi'(\frac{p_i(t)}{\pi_i})}(\phi'(\frac{p_j(t)}{\pi_j})-\phi'(\frac{p_i(t)}{\pi_i}))\\
=&\sum_{j\in N(i)}\omega_{ij}(\frac{p_j(t)}{\pi_j}-\frac{p_i(t)}{\pi_i}),
\end{split}
\end{equation}
where we use the fact that $\theta_{ij}(p)=\theta(\frac{p_i}{\pi_i}, \frac{p_j}{\pi_j})=\frac{\frac{p_j}{\pi_j}-\frac{p_i}{\pi_i}}{\phi'(\frac{p_j}{\pi_j})-\phi'(\frac{p_i}{\pi_i})}$ and $\theta_{ij}(p)(\phi'(\frac{p_j}{\pi_j})-\phi'(\frac{p_i}{\pi_i}))=\frac{p_j}{\pi_j}-\frac{p_i}{\pi_i}$. 

In fact, gradient flow equation \eqref{GD} is a Kolmogorov forward equation for a time-continuous reversible Markov chain. We need to exchange notations in reversible Markov chains and finite weighted graphs $G=(I, E, \omega)$. In other words, denote 
\begin{equation}\label{Q}
Q_{ij}:=\begin{cases}
\frac{\omega_{ij}}{\pi_i} & \textrm{if $j\neq i$;}\\
-\sum_{k\in N(i)}\frac{\omega_{ik}}{\pi_i} & \textrm{if $j=i$}.
\end{cases}
\end{equation}
With this notation equation \eqref{GD} satisfies 
\begin{equation*}
\frac{dp_i(t)}{dt}=\sum_{j=1}^n [Q_{ji}p_j(t)-Q_{ij}p_i(t)]. 
\end{equation*}
The $Q$-matrix is the generator of a reversible Markov chain in $I$. It satisfies the row sum zero condition:
\begin{equation*}
\sum_{j=1}^n Q_{ij} = 0, \qquad Q_{ij}\geq 0, \quad \textrm{for $j\neq i$}.
\end{equation*}
And $\pi=(\pi_i)_{i=1}^n\in \mathbb{R}^n$ defined in \eqref{vol} is an invariant measure for ODE \eqref{GD} with the detailed balance relation
\begin{equation*}
Q_{ij}\pi_i = Q_{ji} \pi_j.
\end{equation*} 
\subsection{Stochastic reversible Markov process}
We next demonstrate that the gradient drift diffusion process in $(\mathcal{P}(I), g^W)$ satisfies a stochastic reversible Markov process on finite states. 

Let $\beta>0$ be a positive scalar. Consider SDE \eqref{GGSDE} as the gradient drift-diffusion flow of $\phi$-divergence in $(\mathcal{P}(I), g^W)$:
\begin{equation}\label{SGD}
\begin{split}
dp_i(t)=&\qquad\sum_{j\in N(i)}\omega_{ij}\theta_{ij}(p(t))(\frac{\partial}{\partial p_j}-\frac{\partial}{\partial p_i})\mathrm{D}_{\phi}(p(t)\|\pi)dt\\
&+\beta\sum_{j\in N(i)}\omega_{ij}\theta_{ij}(p(t))(\frac{\partial}{\partial p_j}-\frac{\partial}{\partial p_i})\log\frac{\mathbb{\Pi}(p(t))^{\frac{1}{2}}}{\theta_{ij}(p(t))}dt\\
&+\sqrt{\beta}\sum_{j\in N(i)}\sqrt{\omega_{ij}\theta_{ij}(p(t))} (dB_{ij}(t)-dB_{ji}(t)).
\end{split}
\end{equation}
{We next study several properties of SDE \eqref{SGD} on a simplex set. We rewrite SDE \eqref{SGD} into the format of Kolmogorov forward equation with a Wasserstein common noise perturbation. 
From equation \eqref{GD} and the definition of $Q$-matrix in \eqref{Q}, SDE \eqref{SGD} can be written as follows: }
\begin{equation*}
\begin{split}
dp_i(t)=&\qquad\sum_{j=1}^n [Q_{ji}p_j(t)-Q_{ij}p_i(t)]dt\\
&+\beta\sum_{j\in N(i)}\omega_{ij}\theta_{ij}(p(t))(\frac{\partial}{\partial p_j}-\frac{\partial}{\partial p_i})\log\frac{\mathbb{\Pi}(p(t))^{\frac{1}{2}}}{\theta_{ij}(p(t))}dt\\
&+\sqrt{\beta}\sum_{j\in N(i)}\sqrt{\omega_{ij}\theta_{ij}(p(t))} (dB_{ij}(t)-dB_{ji}(t)).
\end{split}
\end{equation*}
{Recall that the $Q$-matrix represents the classical probability transition rate between nodes in $I$ for a Markov process.  Following the above reformulation, we can define a Wasserstein diffusion perturbed $Q$-matrix. It represents a transition-rate matrix between nodes in $I$, adding with a probability density dependent coefficient Brownian motion. In particular, the diffusion coefficient comes from the metric $g^W$. } 

\begin{definition}[Wasserstein $Q$-matrix]
Assume that $p_i>0$ for all $i\in I$. Define a matrix function $Q^W=(Q^W_{ij})_{1\leq i,j\leq N}\in \mathbb{R}^{n\times n}$, where $Q^W_{ij}\colon\mathbb{R}^n\times \mathbb{R}\times \mathbb{R}^{n\times n}\rightarrow\mathbb{R}$, such that 
\begin{equation*}
Q^W_{ij}(p, \beta, \dot B):=\begin{cases}
Q_{ij}+a_{ij}(p) & \textrm{if $j\neq i$;}\\
-\sum_{k\in N(i)}(Q_{ik}+a_{ik}(p))& \textrm{if $j=i$},
\end{cases}
\end{equation*}
where 
\begin{equation*}
a_{ij}(p):=\frac{1}{p_i}\max\{0, A_{ji}(p)\},
\end{equation*}
and 
\begin{equation*}
A_{ij}(p):=\beta\omega_{ij}\theta_{ij}(p)(\frac{\partial}{\partial p_j}-\frac{\partial}{\partial p_i})\log\frac{\mathbb{\Pi}(p)^{\frac{1}{2}}}{\theta_{ij}(p)}+\sqrt{\beta\omega_{ij}\theta_{ij}(p)} (\dot B_{ij}(t)-\dot B_{ji}(t)).
\end{equation*}
From now on, we call $Q^{W}$ the {\em Wasserstein $Q$-matrix}. 
\end{definition}
Using the matrix function $Q^W$, we rewrite SDE \eqref{SGD} as follows. 
\begin{proposition}
SDE \eqref{SGD} satisfies 
\begin{equation*}
\dot p_i(t)=\sum_{j=1}^n [Q^W_{ji}(p(t), \beta, \dot B)p_j-Q^W_{ij}(p(t), \beta, \dot B)p_i]. 
\end{equation*}
In addition, $Q^W$ satisfies the row sum zero condition:
\begin{equation*}
\sum_{j=1}^n Q^W_{ij}(p,\beta, \dot B) = 0, \qquad Q^W_{ij}(p, \beta, \dot B)\geq 0, \quad \textrm{for $j\neq i$}.
\end{equation*}
If $\beta=0$, then the Wasserstein $Q$-matrix recovers the $Q$-matrix. I.e., 
\begin{equation*}
Q^W(p,0,\dot B)=Q. 
\end{equation*}
\end{proposition}
\begin{proof}
We check that 
\begin{equation*}
\begin{split}
Q^W_{ji}(p)p_j-Q^W_{ij}(p)p_i=&Q_{ji}p_j-Q_{ij}p_i+\frac{1}{p_j}\max\{0, A_{ij}(p)\}p_j-\frac{1}{p_i}\max\{0, A_{ji}(p)\}p_i\\
=&Q_{ji}p_j-Q_{ij}p_i+\max\{0, A_{ij}(p)\}-\max\{0, A_{ji}(p)\}. 
\end{split}
\end{equation*}
From the fact that $A_{ij}(p)=-A_{ji}(p)$, we have 
\begin{equation*}
A_{ij}(p)=\max\{0, A_{ij}(p)\}-\max\{0, -A_{ij}(p)\}=\max\{0, A_{ij}(p)\}-\max\{0, A_{ji}(p)\}. 
\end{equation*}
This finishes the proof. 
\end{proof}
{From above proposition, we note that $Q_{ji}^W$ represents the stochastic transition rate jumping from node $j$ to node $i$. The stochastic perturbation comes from the Wasserstein common noise.}

We last demonstrate the Fokker-Planck equations for SDE \eqref{SGD}. We also present an invariant distribution of SDE \eqref{SGD}.  
\begin{proposition}[Functional Fokker-Planck equations in finite state Wasserstein space]
Denote $\mathbb{P}(t, p)$ as the solution of the probability density function of SDE \eqref{SGD}. 
Then 
\begin{equation}\label{QFPE}
\begin{split}
\frac{\partial}{\partial t}\mathbb{P}(t,p)+\nabla_p\cdot(\mathbb{P}(t,p) (\sum_{j=1}^n [Q_{ji}p_j-Q_{ij}p_i] )_{i=1}^n)=\beta\mathsf{L}^*_W\mathbb{P}(t,p).
\end{split}
\end{equation}
Assume that $Z=\int_{\mathcal{P}(I)}e^{-\frac{1}{\beta}\mathrm{D}_\phi(p\|\pi)}\mathbb{\Pi}(p)^{-\frac{1}{2}}dp<+\infty$, then the stationary solution of equation \eqref{QFPE} satisfies 
\begin{equation*}
\mathbb{P}^*(p)=\frac{1}{Z}e^{-\frac{1}{\beta}\mathrm{D}_\phi(p\|\pi)}\mathbb{\Pi}(p)^{-\frac{1}{2}}.
\end{equation*}
\end{proposition}
\begin{proof}
The proof directly follows from Proposition \ref{prop2}. We have
\begin{equation*}
\begin{split}
\frac{\partial\mathbb{P}(t,p)}{\partial t}=&-\nabla_p\cdot(\mathbb{P}(t,p) (\sum_{j=1}^n [Q_{ji}p_j-Q_{ij}p_i] )_{i=1}^n)+\beta\mathsf{L}^*_W\mathbb{P}(t,p)\\
=&\nabla_p\cdot(\mathbb{P}(t,p)L(p)\nabla_p\mathrm{D}_\phi(p\|\pi))+\beta {L}^*_W\mathbb{P}(t,p)\\
=&\nabla_p\cdot(\mathbb{P}(t,p)L(p)\nabla_p\mathrm{D}_\phi(p\|\pi))+\beta \nabla_p\cdot(\mathbb{P}(t,p)L(p)\nabla_p\log\frac{\mathbb{P}(t,p)}{\mathbb{\Pi}(p)^{-\frac{1}{2}}})\\
=&\beta\nabla_p\cdot(\mathbb{P}(t,p)L(p)\nabla_p\log\frac{\mathbb{P}(t,p)}{e^{-\frac{1}{\beta}\mathrm{D}_\phi(p\|\pi)}\mathbb{\Pi}(p)^{-\frac{1}{2}}}).
\end{split}
\end{equation*}
Clearly, the stationary density of equation \eqref{QFPE} satisfies 
\begin{equation*}
\mathbb{P}^*(p)=\frac{1}{Z}e^{-\frac{1}{\beta}\mathrm{D}_\phi(p\|\pi)}\mathbb{\Pi}(p)^{-\frac{1}{2}},
\end{equation*}
where $Z<+\infty$ is a normalization constant.
\end{proof}

{We remark that the finite state Wasserstein drift diffusion defines a class of diffusion processes on the simplex set. They add a particular class of probability dependent Brownian motions into transition kernels of Markov processes. These noises are built from gradient structures of Markov processes, and essentially form ``canonical'' noises in the probability manifold. The proposed stochastic process can be viewed as the finite state analog of super Brownian motion, studied in \cite{KLR, WD1, WD}. In the future work, we shall investigate physical modelings and applications of Wasserstein diffusion processes on graphs. We expect geometric calculations in the probability simplex play essential roles in constructing and understanding the proposed stochastic Markov processes.}

\section{Examples on a two point space}\label{section4}
In this section, we present several examples of Wasserstein gradient drift diffusion processes \eqref{SGD} on a two-point state. 

Consider a two-point graph $I=\{1,2\}$, with $\omega_{12}=\omega_{21}>0$, $\omega_{11}=\omega_{22}=0$, and $\pi_1=\pi_2=\frac{1}{2}$.  
Denote $p=(p_1, p_2)^{\ts}\in \mathcal{P}(I)\subset \mathbb{R}^2$ as the probability function. In this case,  
\begin{equation*}
L(p)=\begin{pmatrix}
\theta_{12}(p)\omega_{12} & -\theta_{12}(p)\omega_{12}\\
-\theta_{12}(p)\omega_{12} & \theta_{12}(p)\omega_{12}  
\end{pmatrix}.
\end{equation*}
The eigenvalue of $L(p)$ can be computed explicitly. In other words, 
\begin{equation*}
\mathbb{\Pi}(p)=\lambda_1(p)=2\omega_{12}\theta_{12}(p). 
\end{equation*}
The Wasserstein gradient drift-diffusion \eqref{GGSDE} satisfies 
 \begin{equation}\label{SDE_example}
\left\{\begin{aligned}
&dp_1(t) =\omega_{12}\theta_{12}(p(t))(\frac{\partial}{\partial p_2}-\frac{\partial}{\partial p_1})[\mathbb{V}(p)-\frac{\beta}{2}\log\theta_{12}(p(t))]dt\\
&\hspace{1.5cm}+\sqrt{\beta\omega_{12}\theta_{12}(p(t))} (dB_{12}(t)-dB_{21}(t)), \\
&dp_2(t)=\omega_{12}\theta_{12}(p(t))(\frac{\partial}{\partial p_1}-\frac{\partial}{\partial p_2})[\mathbb{V}(p)-\frac{\beta}{2}\log\theta_{12}(p(t))]dt\\
&\hspace{1.5cm}+\sqrt{\beta\omega_{12}\theta_{12}(p(t))} (dB_{21}(t)-dB_{12}(t)), 
\end{aligned}\right.
\end{equation}
where $(B_{12}, B_{21})\in\mathbb{R}^2$ are standard independent Brownian motions. 

The two dimensional SDE \eqref{GGSDE} can be further simplified into a one dimensional equation. Denote $x(t): =p_1(t)\in [0,1]$, $p_2(t)=1-x(t)$, $h=\sqrt{\omega_{12}}>0$, $V(x):=\mathbb{V}(p)=\mathbb{V}(x,1-x)$, and $\theta(x):=\theta_{12}(p)$. Note that $\frac{\partial}{\partial p_1}\theta_{12}(p)-\frac{\partial}{\partial p_2}\theta_{12}(p)=\frac{d}{dx} \theta(x)=\theta'(x)$, and $\frac{\partial}{\partial p_1}\mathbb{V}(p)-\frac{\partial}{\partial p_2}\mathbb{V}(p)=\frac{d}{dx}V(x)=V'(x)$. Write $B(t)=\frac{1}{\sqrt{2}}(B_{12}(t)-B_{21}(t))$. Then SDE \eqref{SDE_example} satisfies 
\begin{equation}\label{SDE2}
dx_t=-h^2[\theta(x_t)V'(x_t)-\frac{\beta}{2}\theta'(x_t)]dt+h\sqrt{2\beta\theta(x_t)}dB_t, 
\end{equation}
where $x_t\in [0,1]$ is the solution. Thus the Fokker-Planck equation of SDE \eqref{SDE2} satisfies 
\begin{equation*}
\begin{split}
\partial_t\rho(t,x)=&h^2\partial_x(\rho(t,x)[\theta(x)V'(x)-\frac{\beta}{2}\theta'(x)])+\beta h^2\partial_{xx}(\rho(t,x)\theta(x))\\
=&\beta h^2\partial_x\Big(\rho(t,x)\theta(x)\partial_x\log\frac{\rho(t,x)}{e^{-\frac{1}{\beta}V(x)}\theta(x)^{-\frac{1}{2}}}\Big).
\end{split}
\end{equation*}
And the stationary density of SDE \eqref{SDE2} satisfies 
\begin{equation*}
\rho^*(x)=\frac{1}{Z}e^{-\frac{V(x)}{\beta}}\theta(x)^{-\frac{1}{2}},
\end{equation*}
where we assume that $Z=\int_0^1e^{-\frac{V(y)}{\beta}}\theta(y)^{-\frac{1}{2}} dy<+\infty$.
\begin{example}[Wasserstein common noises on a two point space]\label{ex5}
Let $\beta=1$ and $\mathbb{V}(p)=0$. The SDE \eqref{SDE2} forms the canonical Wasserstein common noise:  
\begin{equation*}
dx_t =\frac{h^2}{2}\theta'(x_t)dt+h\sqrt{2\theta(x_t)}dB_t.
\end{equation*}
In this case, assume that $Z=\int_0^1\theta(y)^{-\frac{1}{2}} dy<+\infty$, the stationary density in simplex set satisfies 
\begin{equation*}
\rho^*(x)=\frac{1}{Z}\theta(x)^{-\frac{1}{2}}.
\end{equation*}
In particular, let $\theta$ be a geometric mean, i.e., $\theta(x)=2\sqrt{x(1-x)}$. Then SDE \eqref{SDE2} forms 
\begin{equation*}
dx_t=h^2\frac{1-2x_t}{x_t^{\frac{1}{2}}(1-x_t)^{\frac{1}{2}}}dt+ 2hx_t^{\frac{1}{4}}(1-x_t)^{\frac{1}{4}}dB_t.
\end{equation*}
We simulate the above SDE numerically in the time interval $[0,1]$ by the Euler--Maruyama scheme, for parameters $h=0.1$, $t\in [0,1]$, $x_0=0.5$.   
\begin{figure}[h]
 \subfloat[Trajectories of Wasserstein common noises $x_t$.]
{\includegraphics[width=0.5\textwidth]{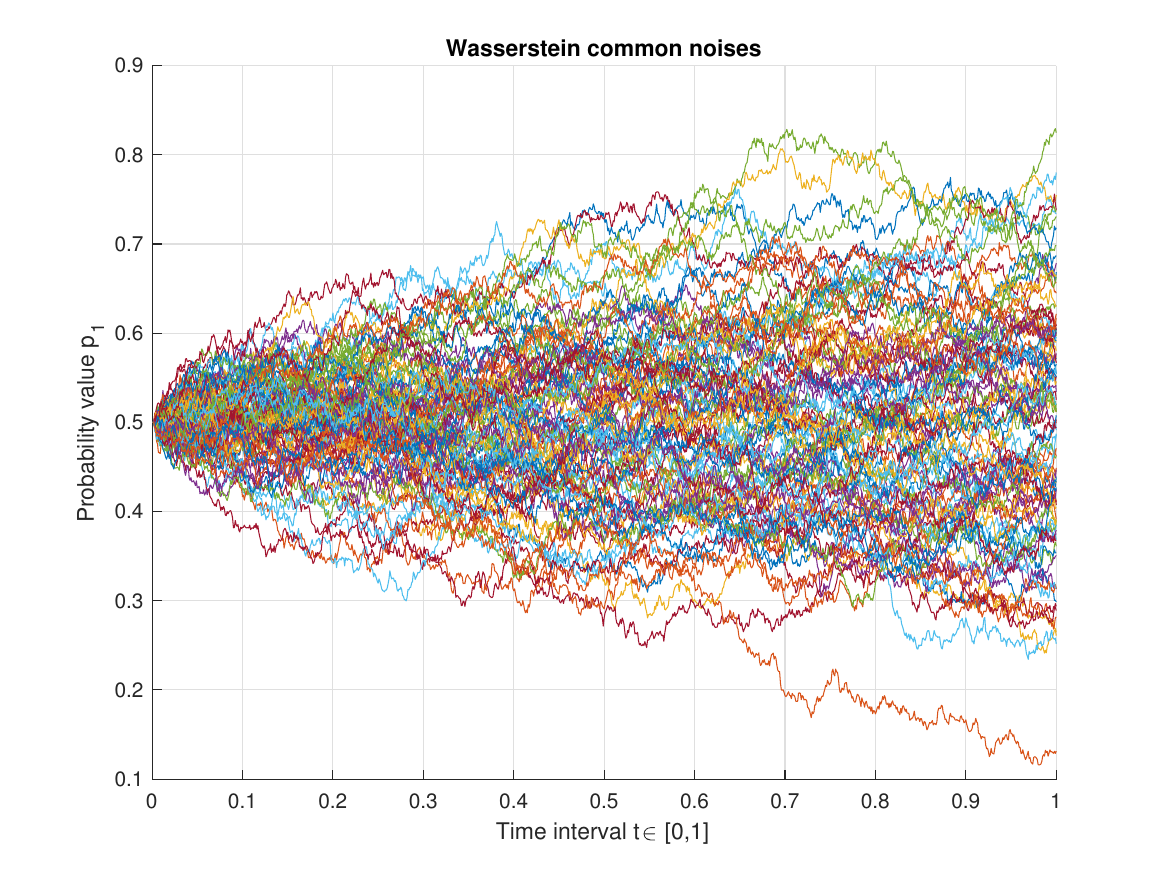}}
 \subfloat[Stationary density $\rho^*$.]
{\includegraphics[width=0.5\textwidth]{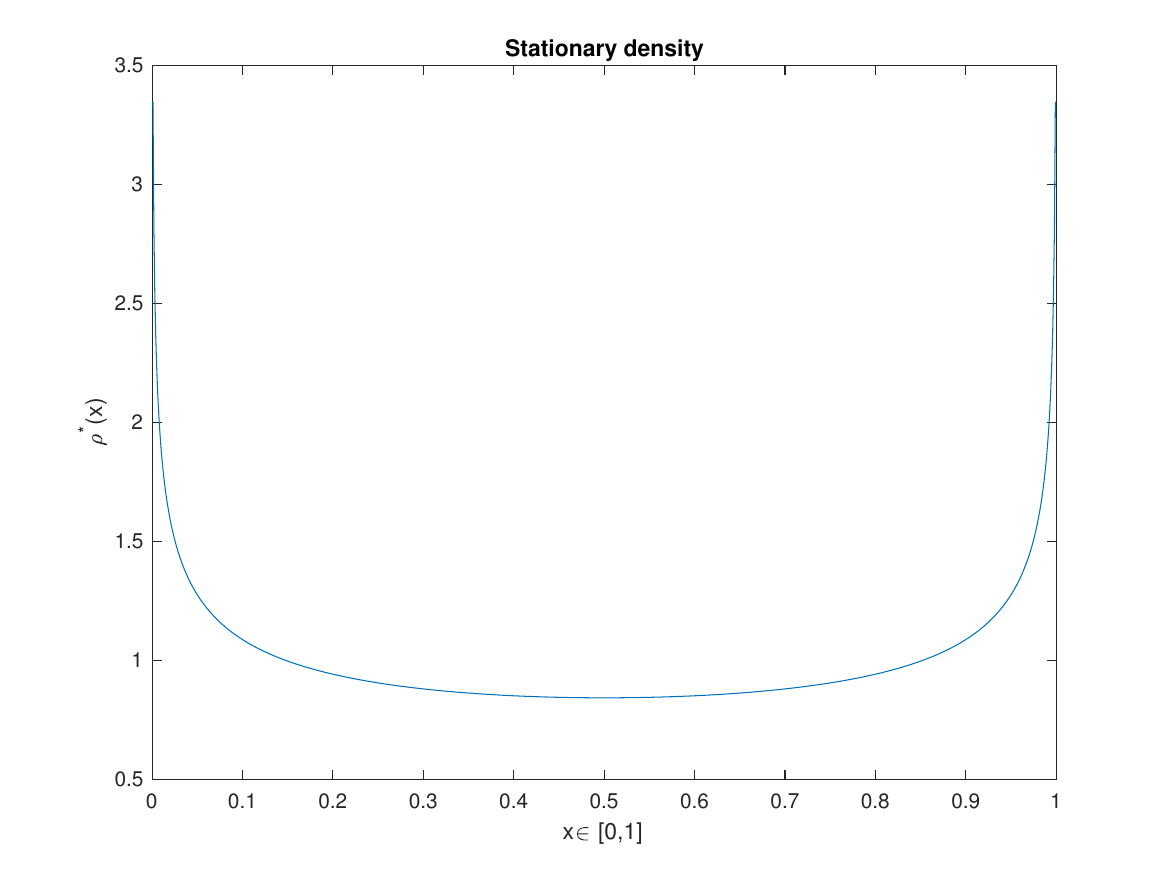}}
\caption{Simulations of Wasserstein common noises with the geometric mean function in Example \ref{ex5}.}
\end{figure} 

\end{example}

\begin{example}[Individual and Wasserstein common noises on a two point space]\label{ex6}
Let $\beta=1$ and $\mathbb{V}(p)=p_1\log\frac{p_1}{\pi_1}+p_2\log \frac{p_2}{\pi_2}$, i.e., $V(x)=x\log x+(1-x)\log(1-x)+\log 2$. 
Then SDE \eqref{SDE2} satisfies 
\begin{equation*}
dx_t =h^2[-\theta(x_t)(\log x_t-\log (1-x_t))+\frac{1}{2}\theta'(x_t)]dt+h\sqrt{2\theta(x_t)}dB_t.
\end{equation*}
 And the stationary density in simplex set satisfies 
\begin{equation*}
\rho^*(x)=\frac{1}{Z}(\frac{1}{x})^x(\frac{1}{1-x})^{1-x}\theta(x)^{-\frac{1}{2}}, \quad Z=\int_0^1(\frac{1}{y})^y(\frac{1}{1-y})^{1-y}\theta(y)^{-\frac{1}{2}} dy<+\infty.
\end{equation*}
In particular, let $\theta$ be a logarithm mean, i.e., $\theta(x)=\frac{2(2x-1)}{\log x-\log(1-x)}$. Then  SDE \eqref{SDE2} forms 
\begin{equation*}
\begin{split}
dx_t=&h^2[2(1-2x_t)+\frac{(1-2x_t)}{(x_t-x_t^2)(\log x_t-\log(1-x_t))^2}+\frac{2}{(\log x_t-\log(1-x_t))}]dt\\
&+ 2h\sqrt{\frac{2x_t-1}{\log x_t-\log(1-x_t)}}dB_t.
\end{split}
\end{equation*}
Again, we simulate the above SDE numerically in the time interval $[0,1]$ by the Euler--Maruyama scheme, for parameters $h=0.1$, $t\in [0,1]$, $x_0=0.5$.   
\begin{figure}[h]
 \subfloat[Trajectories of individual noises and Wasserstein common noises $x_t$.]
{\includegraphics[width=0.5\textwidth]{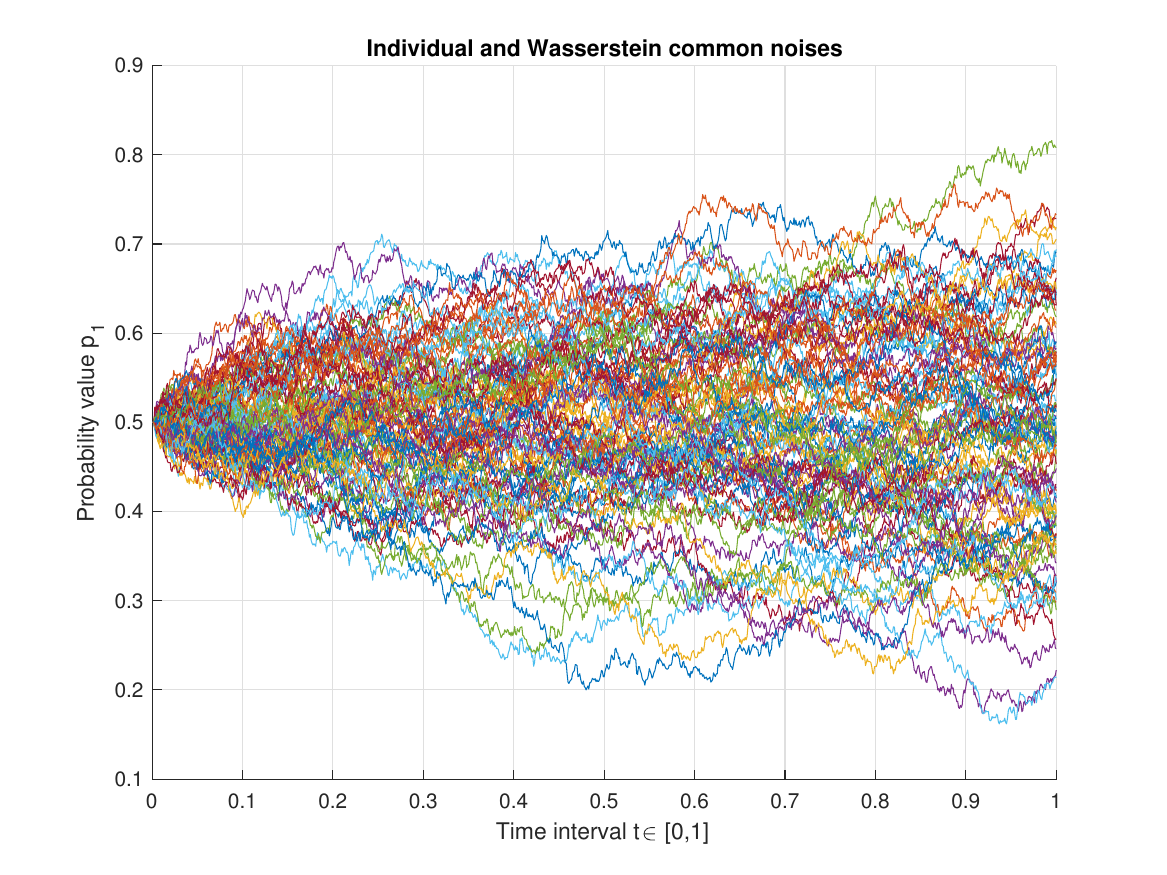}}
 \subfloat[Stationary density $\rho^*$.]
{\includegraphics[width=0.5\textwidth]{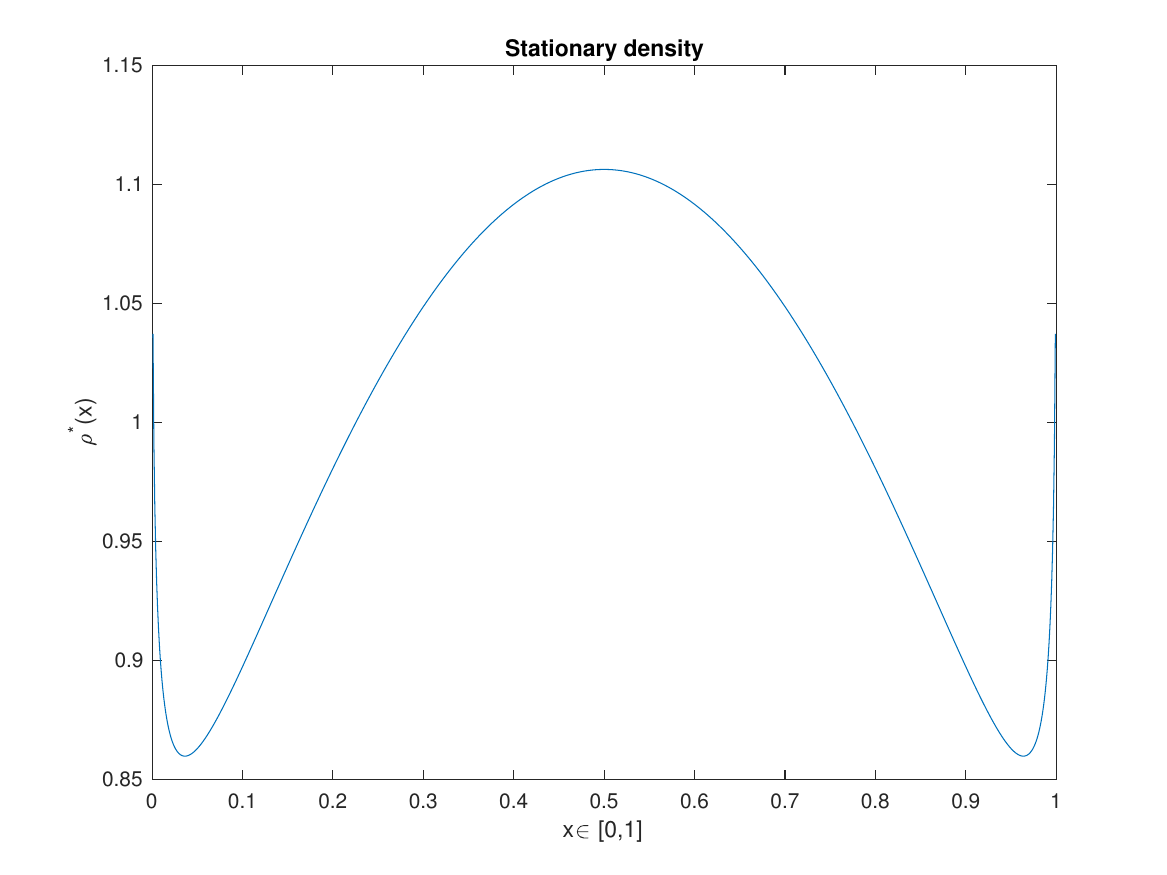}}
\caption{Simulations of individual and Wasserstein common noises with the logarithm mean function in Example \ref{ex6}.}
\end{figure} 

\end{example}

\section{Discussions}
In this paper, we present Wasserstein common noises in the probability simplex set. They are constructed from the Laplace-Beltrami operator in finite state Wasserstein space. We also derive a drift-diffusion process in the probability simplex. Extending equivalence relationships between gradient flows and reversible Markov processes, we introduce a class of stochastic reversible Markov processes. The stochastic perturbation is added from the canonical Wasserstein common noise on finite states. {In this procedure, one can define a class of stochastic Markov processes, where the Brownian motion perturbation is added into the transition kernel of Markov processes. These processes are discrete counterparts of super Brownian motions studied in mathematical physics communities \cite{KLR, WD1, WD}.} 

In particular, equation \eqref{GD} is known as the strong Onsager gradient flow \cite{GLL, ON}:
\begin{equation*}
\frac{dp}{dt}=-L(p)\nabla_p\mathrm{D}_\phi(p\|\pi)=\mathrm{div}_\omega(\theta(p_t) \nabla_\omega\nabla_p\mathrm{D}_\phi(p\|\pi)),
\end{equation*}
where $\nabla_p{D}_\phi(p\|\pi)$ is the generalized force and $L(p)$ is the Onsager response matrix. In this sense, the proposed SDE \eqref{SGD} satisfies {\em Onsager gradient drift diffusions}:
\begin{equation*}
\begin{split}
dp_t=&\mathrm{div}_\omega(\theta(p_t) \nabla_\omega\nabla_p\mathrm{D}_\phi(p_t\|\pi))dt+\beta\mathrm{div}_\omega(\theta(p_t) \nabla_\omega \nabla_p\log\frac{\mathbb{\Pi}(p_t)^{\frac{1}{2}}}{\theta(p_t)})dt+\sqrt{2\beta}\mathrm{div}_\omega(\sqrt{\theta(p_t)}dB^E_t),\\
& \hspace{1cm}\textrm{Individual noises} \hspace{4cm}\textrm{Wasserstein common noises}
\end{split}
\end{equation*}
where $\beta>0$ is a scalar for the canonical Wasserstein diffusion and $\mathbb{\Pi}(p_t)$ is the product of positive eigenvalues of the Onsager response matrix $L(p_t)$.  {We remark that the Onsager response matrix $L(p)$ also introduces a natural class of stochastic Markov processes, i.e., drift--diffusion processes in the probability manifold.}

In future work, we shall investigate properties of Wasserstein drift-diffusion processes on discrete states. In particular, the boundary set and corners of the probability simplex bring difficulties in the existence of strong solutions of SDEs \eqref{GGSDE}. One has to assume that the square root of activation function $\sqrt{\theta}$ is Lipschitz, which is often not satisfied for many divergence induced activation functions. The other interesting question is about the entropy dissipation analysis for probability density function supported on a simplex set; see \cite{LiG}. In applications, we remark that finite states Wasserstein drift-diffusion processes are essential in modeling and computations of population games in social dynamics. Typical examples include stochastic evolutionary dynamics \cite{FY}, mean field games \cite{cardaliaguet2019master, GLL}, and estimation problems in data sciences \cite{WWG}. More importantly, we shall develop fast and accurate algorithms to compute and model Wasserstein drift diffusion processes arised in social sciences, biology, evolutionary game theory, and Bayesian and AI sampling problems. {We expect that geometric calculations in probability manifolds are essential tools in these modeling, computation and analysis problems. }

\noindent\textbf{Acknowledgement:} On behalf of all authors, the corresponding author states that there is no conflict of interest.

\bibliographystyle{abbrv}

\section*{Appendix: Derivations of $L^*_W$ and $L_W$}
The derivation of Fokker-Planck equation for SDE \eqref{GGSDE} is standard. We omit it here. 
We only show the derivation from the Wasserstein Laplacian-Beltrami operator on simplex set to the Kolmogorov forward and backward operators in a simplex set.  
 \begin{proof} 
The Laplace--Beltrami operator in $(\mathcal{P}(I), g^W)$ satisfies
\begin{equation}\label{heat}
\begin{split}
\Delta_{W}\mathbb{p}(p)=&\mathbb{\Pi}(p)^{\frac{1}{2}}\nabla_p\cdot\Big(\mathbb{\Pi}(p)^{-\frac{1}{2}}L(p) \nabla_p\mathbb{p}(p)\Big),
\end{split}
\end{equation}
where $\mathbb{p}\in C^{\infty}(\mathcal{P}(I); \mathbb{R})$ is a probability density function on simplex set w.r.t. $\textrm{vol}_W$. Here 
\begin{equation*}
\int_{\mathcal{P}(I)}\mathbb{p}(p)d\textrm{vol}_W(p)=1.
\end{equation*}
Denote a probability density function of simplex set w.r.t. Lebesgue measure in $\mathbb{R}^n$ as
$$\mathbb{P}(p)=\mathbb{p}(p) \mathbb{\Pi}(p)^{-\frac{1}{2}}.$$ Then operator \eqref{heat} forms
\begin{equation*}
\begin{split}
L_{W}^*\mathbb{P}(p)=&\nabla_p\cdot\Big(\mathbb{\Pi}(p)^{-\frac{1}{2}}L(p) \nabla_p\mathbb{p}(p)\Big)\\
=&\nabla_p\cdot\Big(\mathbb{p}(p)\mathbb{\Pi}(p)^{-\frac{1}{2}}L(p) \nabla_p\log\mathbb{p}(p)\Big)\\
=&\nabla_p\cdot \Big(\mathbb{P}(p) L(p)\nabla_p\log \frac{\mathbb{P}(p)}{\mathbb{\Pi}(p)^{-\frac{1}{2}}}\Big)\\
=&\frac{1}{2}\nabla_p\cdot \Big(\mathbb{P}(p) L(p) \nabla_p\log\mathbb{\Pi}(p)\Big)+\nabla_p\cdot \Big(L(p) \nabla_p\mathbb{P}(p)\Big),
\end{split}
\end{equation*}
where we use the fact 
\begin{equation*}
\nabla_p \mathbb{p}(p)=\mathbb{p}(p)\nabla_p\log\mathbb{p}(p),\quad \nabla_p\mathbb{P}(p)=\mathbb{P}(p)\nabla_p\log\mathbb{P}(p). 
\end{equation*}
In details, we have
\begin{equation*}
\begin{split}
\nabla_p\cdot \Big(L(p) \nabla_p\mathbb{P}(p)\Big)=&\sum_{i=1}^n\sum_{j=1}^n\frac{\partial}{\partial p_j}\Big(L(p)_{ij}\frac{\partial}{\partial p_i}\mathbb{P}(p)\Big)\\
=&\sum_{i=1}^n\sum_{j=1}^n\Big(\frac{\partial}{\partial p_j}L(p)_{ij}\frac{\partial}{\partial p_i}\mathbb{P}(p)+L(p)_{ij}\frac{\partial^2}{\partial p_i\partial p_j}\mathbb{P}(p)\Big)\\
=&\sum_{i=1}^n\sum_{j\in N(i)}\omega_{ij}(\frac{\partial}{\partial p_i}-\frac{\partial}{\partial p_j})\theta_{ij}(p)\frac{\partial}{\partial p_i}\mathbb{P}(p)\\
&-\sum_{i=1}^n\sum_{j\in N(i)}\omega_{ij}\theta_{ij}(p)\frac{\partial^2}{\partial p_i\partial p_j}\mathbb{P}(p)+\sum_{i=1}^n\sum_{k\in N(i)}\omega_{ik}\theta_{ik}(p)\frac{\partial^2}{\partial p_i\partial p_i}\mathbb{P}(p)\\
=&\quad\frac{1}{2}\sum_{(i,j)\in E}\omega_{ij}(\frac{\partial}{\partial p_i}-\frac{\partial}{\partial p_j})\theta_{ij}(p)(\frac{\partial}{\partial p_i}-\frac{\partial}{\partial p_j})\mathbb{P}(p)\\
&+\frac{1}{2}\sum_{(i,j)\in E}\omega_{ij}\theta_{ij}(p)(\frac{\partial^2}{\partial p_j\partial p_j}+\frac{\partial^2}{\partial p_i\partial p_i}-2\frac{\partial^2}{\partial p_i\partial p_j})\mathbb{P}(p).
\end{split}
\end{equation*}
In above derivation, we use the fact that 
\begin{equation*}
\begin{split}
\sum_{j=1}^n\frac{\partial}{\partial p_j}L(p)_{ij}=&\sum_{j\neq i}\frac{\partial}{\partial p_j}L(p)_{ij}+\frac{\partial}{\partial p_i}L(p)_{ii}\\
=&-\sum_{j\in N(i)}\omega_{ij}\frac{\partial}{\partial p_j}\theta_{ij}(p)+\sum_{k\in N(i)}\omega_{ki}\frac{\partial}{\partial p_i}\theta_{ki}(p)\\
=&\sum_{j\in N(i)}\omega_{ij}(\frac{\partial}{\partial p_i}-\frac{\partial}{\partial p_j})\theta_{ij}(p). 
\end{split}
\end{equation*}
Similarly, we can derive 
$$\nabla_p\cdot \Big(\mathbb{P}(p) L(p) \nabla_p\log\mathbb{\Pi}(p)\Big)=(\nabla_p\mathbb{P}(p), L(p)\nabla_p\log\mathbb{\Pi}(p))+\mathbb{P}(p)\nabla_p\cdot \Big(L(p) \nabla_p\log\mathbb{\Pi}(p)\Big).$$ 
This finishes the proof. 

We next derive the Kolmogorov backward operator for SDE \eqref{GSDE}.  
\begin{equation*}
\int_{\mathcal{P}(I)}\mathbb{\Phi}(p)\mathsf{L}^*_W\mathbb{P}(p)dp=\int_{\mathcal{P}(I)}\mathbb{P}(p)\mathsf{L}_W\mathbb{\Phi}(p)dp.
\end{equation*}
Clearly, we have 
\begin{equation*}
\begin{split}
\mathsf{L}_W\mathbb{\Phi}(p)=&-\frac{1}{2}(\nabla_p\mathbb{\Phi}(p), L(p)\nabla_p\log\mathbb{\Pi}(p))+\nabla_p\cdot(L(p)\nabla_p\mathbb{\Phi}(p)).
\end{split}
\end{equation*}
This finishes the proof.
\end{proof}

\end{document}